\documentclass{article}

\usepackage[utf8]{inputenc} 
\usepackage{amsmath,amsthm,amssymb}
\usepackage{cases}
\usepackage{mathrsfs} 
\usepackage{rsfso} 
\usepackage{ascmac} 
\usepackage{fancybox} 
\usepackage{ulem} 
\usepackage{authblk} 
\usepackage[bbgreekl]{mathbbol}   
\usepackage{physics}  
\usepackage{bm} 
\usepackage{comment}
\usepackage{dashbox} 
\usepackage{csquotes} 

\usepackage[dvipdfmx]{graphicx}
\usepackage[dvipdfmx]{xcolor}
\usepackage{tikz}

\def\rnum#1{\expandafter{\romannumeral #1}}
\def\Rnum#1{\uppercase\expandafter{\romannumeral #1}} 

\theoremstyle{plain}
\newtheorem{thm}{Theorem}[section]
\newtheorem{pro}{Proposition}[section]

\newtheorem{lem}{Lemma}[section]
\newtheorem{rem}{Remark}[section]
\newtheorem{Def}{Definition}[section]

\newcommand{\Rn}{\mathbb{R^{\textit{n}}}}
\newcommand{\SP}{\mathbb{S}}
\newcommand{\N}{\mathbb{N}}

\newcommand{\R}{\mathbb{R}}
\newcommand{\C}{\mathbb{C}}

\newcommand{\dint}{\displaystyle \int}

\newcommand{\dis}{\displaystyle}
\newcommand{\vphi}{\varphi}
\newcommand{\vep}{\varepsilon}

\newcommand{\med}{\mathbf{Med}}

\newcommand{\lan}{\langle}
\newcommand{\ran}{\rangle}

\newcommand{\llq}{\lq \lq}
\newcommand{\rrq}{\rq \rq}

\newcommand{\ho}{\"} 
\newcommand{\wti}{\widetilde}

\newcommand{\tsupp}{\textnormal{supp}\,}
\newcommand{\tnormal}{\textnormal}
\newcommand{\tn}{\textnormal}

\newcommand{\ul}{\underline}
\newcommand{\nn}[1]{{\left\vert\kern-0.25ex\left\vert\kern-0.25ex\left\vert #1 
    \right\vert\kern-0.25ex\right\vert\kern-0.25ex\right\vert}} 

\DeclareMathOperator*{\ei}{ess-inf}

\begin{document}

\title{Fractional medians and their maximal functions}
\author{Yohei Tsutsui \footnote{Department of Mathematics, Graduate School of Science, Kyoto University, Kyoto 606-8502, Japan. \newline e-mail: \texttt{ytsutsui@math.kyoto-u.ac.jp}}}
\date{}
\maketitle

\begin{abstract}
In this article, we introduce the fractional medians, give an expression of the set of all fractional medians in terms of non-increasing rearrangements and then investigate mapping properties of the fractional maximal operators defined by such medians.
The maximal operator is a generalization of that in Strömberg \cite{St}.
It turns out that our maximal operator is a more smooth operator than the usual fractional maximal operator.
Further, we give another proof of the embedding from $BV$ to $L^{n/(n-1),1}$ due to Alvino \cite{A} by using the usual medians.
\end{abstract}

{\bf Keywords} Fractional medians,\ Fractional maximal operator,\ Non-increasing rearrangements.

{\bf 2020 Mathematics Subject Classification} 42B02, 46E30, 46E35.

\section{Introduction}
The integral average of function $f$, defined on $\Rn$;
\[
\lan f \ran_E := \dfrac{1}{|E|} \int_E |f| dx,
\]
for $E \subset \Rn$ with $0<|E|<\infty$, plays an important role in Harmonic analysis.
Of course, if $f$ is not locally integrable, the average is not useful.
On the other hand, for such functions, medians exist and are finite.
For example, the median of $|x|^{-n}$ is finite.
For a real-valued measurable function $f$ and the same $E$ as above and $\alpha \in (0,1)$, we say that a real number $m$ is a \textit{median} of $f$ if it satisfies
\begin{equation} \label{median def}
\left|\left\{x \in E; f(x) < m \right\} \right| \le (1-\alpha) |E| \ \tn{and} \ \left|\left\{x \in E; f(x) > m \right\} \right| \le \alpha |E|.
\end{equation}
Let us denote $\med^\alpha_E[f]$ as the set of all the medians.
The integral average is the mean of function in the $L^1$ sense, whereas the median can be regarded the mean of function in the $L^{1,\infty}$ sense.
The gap of them is very important in Harmonic analysis, especially for boundedness of several operators.
Unfortunately, median is not uniquely determined.
To overcome this difficulty, several authors used the maximal median, $\max \med_E^\alpha[f]$, which is also a median of $f$.

We explain how median have been applied.
Carleson \cite{C} gave a definition of $BMO$, all of functions having bounded mean oscillation, by using the maximal median.
Since his definition does not use the integral average, his definition appears to be weaker than the conventional one.
Strömberg \cite{St} consider the case of critical index that Carleson \cite{C} did not addressed, and then established the equivalence
\[
\|f\|_{BMO} \approx \dis \sup_Q L \left[(f - m_Q)\chi_Q \right] (\alpha |Q|),
\]
where the supremum is taken over all cubes $Q \subset \Rn,\ m_Q \in \med_Q^\alpha[f],\ 0<\alpha \le 1/2$, and $L$ is the left-continuous \textit{non-increasing rearrangement}, see below for the definition.
This equivalence also indicates that the local integrability is not necessary for the definition of $BMO$.
Furthermore, it was mentioned in \cite{St} that this equivalence fails if the rearrangement $L$ is replaced by another rearrangement $R$, which is \textit{right} continuous.
The characteristic function $f(x) = \chi_E(x)$ with $0<|E|$ and $0<|E^c|$ gives a counterexample, indeed $\|\chi_E\|_{BMO} = 1/2$ and
\[
\sup_Q R \left[(\chi_E - (\chi_E)_Q)\chi_Q \right] (|Q|/2) = \infty,
\]
where $(\chi_E)_Q \in \med_{Q}^{1/2}[\chi_E]$.
In \cite{St}, he also introduced the local sharp maximal operator;
\[
M^\sharp_{0,\alpha}f(x) := \sup_{Q \ni x} \inf_{c \in \C} L[f \chi_Q](\alpha |Q|),
\]
and make use of it to give pointwise estimate for the Riesz transform $R_j$;
\begin{equation} \label{sharp maximal}
M^\sharp_{0,\alpha}(R_jf)(x) \lesssim M^\sharp f(x),
\end{equation}
where $M^\sharp$ is the sharp maximal operator.
This operator appears to have been first addressed by John \cite{J}.
Obviously, we have from \eqref{sharp maximal} the inequality $M^\sharp_{0,\alpha}(R_jf)(x) \lesssim M f(x)$.
This can be compared with a well-known inequality; $M^\sharp (R_jf) \lesssim M(|f|^r)^{1/r}$ with any $r \in (1,\infty)$.
This improvement is achieved by using the median instead of the integral mean.
Jawerth-Torchinsky \cite{J-T} provided a representation, using the local sharp maximal operator, of the $K$-functional including $BMO$.
Fujii \cite{F2} proved the following; for almost everywhere $x_0$ that for any sequence of medians $\{m_{Q_j}\}_{j=1}^\infty$ of a measurable function $f$ with cubes $Q_j$ shrinking to $x_0$, the limit $\lim_{j \to \infty} m_{Q_j}$ equals $f(x_0)$.
Poelhuis-Torchinsky \cite{P-T} showed the same fact in the general case.
See 5 in Remark \ref{rem bdd} for this fact.
Lerner \cite{L4} gave a useful decomposition of a measurable function, which is an improvement on previous studies such as Carleson \cite{C2}, Garnett-Jones \cite{G-J}, and Fujii \cite{F1}.
We can find applications of medians in Poincar\'e inequality, see Theorem 3.51 in Ambrosio-Fusco-Pallara \cite{A-F-P} for example.

Here, we introduce the \textit{fractional median}.
The motivation of our definition is to ensure that the inequality
\[
m \le \inf_{r>0} \left(\dfrac{1}{\alpha |Q|} \int_Q |f|^r dx \right)^{1/r},\ \tn{for any}\ m \in \med^\alpha_Q[f]
\]
holds even in the fractional case.
Let $f$ be a real-valued measurable function, $E \subset \Rn$ with $0<|E|<\infty$, $\gamma \in [0,n)$ and $\alpha \in (0,|E|^{\gamma/n})$.
We say a real number $m$ is a fractional median of $f$ if
\begin{align*}
\left|\left\{x \in E; f(x) < m \right\} \right| & \le |E| - \alpha |E|^{1-\gamma/n} \ \tn{and} \\
\left|\left\{x \in E; f(x) > m \right\} \right| & \le \alpha |E|^{1-\gamma/n}.
\end{align*}
We denote the set of all fractional medians by $\med^{\alpha,\gamma}_E[f]$.
While $|B(0,1)|^{\alpha/n} \lan |x|^{-n} \ran_{B(0,1)} = \infty$, its fractional median exists and is given by
\[
\med_{B(0,1)}^{\alpha, \gamma}[|x|^{-n}] = \left\{\alpha^{-1} \right\}.
\]
Furthermore, the fractional median of $|x|^{-m}$ for $m>n$ also exists.
One can see that for any $m \in \med_E^{\alpha,\gamma}[f]$, it holds
\[
m \le \inf_{r>0} \left(\dfrac{|E|^{\gamma/n}}{\alpha |E|} \int_E |f|^r dx \right)^{1/r},
\]
which is showed in Proposition \ref{L}.
This inequality is consistent with the fact that the definition of the median does not require local integrability.
In general, not only in the case $\gamma=0$, the fractional median is not uniquely determined.
One of the purpose of this article is to give a representation of $\med^{\alpha, \gamma}_E[f]$ in terms of rearrangements $R$ and $L$;
\[
\med^{\alpha,\gamma}_E[f] = \left[R[f \chi_E]\left(\alpha |E|^{1-\gamma/n} \right), L[f \chi_E]\left(\alpha |E|^{1-\gamma/n} \right) \right],
\]
see Theorem \ref{express fra med}.
This means that the both ends are the fractional medians of $f$, and the right end is the fractional maximal median.
Hence, the medians are quantitatively controlled by these rearrangements from both below and above.
Poelhuis-Torchinsky \cite{P-T} showed the inclusion \llq $\subset$\rrq in the case $\gamma=0$.

Our next interest lies on the mapping properties of the maximal function;
\[
m_{\alpha, \gamma}^L[f](x) := \dis \sup_Q L[f \chi_Q]\left(\alpha |Q|^{1-\gamma/n} \right) \chi_Q(x).
\]
See Section 4.
This operator with $\gamma=0$ was introduced by Strömberg \cite{St}.
It is found that $m^L_{\alpha,0}$ does not cause significant quantitative changes, see Remark 4.4.
It is natural to ask the mapping property of $m^L_{\alpha,\gamma}$.
Although the fractional maximal operator
\[
M_\gamma f(x) := \sup_Q |Q|^{\gamma/n} \lan f \ran_Q \chi_Q(x),\ (0<\gamma<n)
\]
maps $L^p$ to $L^{\wti{p}}$ for all $p \in (1,n/\gamma)$ with $\wti{p} = np/(n-\gamma)$, our maximal operator $m^L_{\alpha,\gamma}$ maps $L^{p,q}$ to $L^{\wti{p},q}$ for any $p,\ q \in (0,\infty]$ and $\wti{p} \in [np/(n-\gamma),\infty)$.
See Theorem \ref{maximal} in Section 4.
The proof follows from a pointwise estimate of the distribution function.
While the maximal operator $m^L_{\alpha,0}$ does not cause significant changes in functions, the operator provides a bridge between $f$ and $M(\chi_{\{|f| > \lambda \}})$ in terms of the distribution function scale.
By applying this fact, we can offer another proof of Alvino's embedding $BV \hookrightarrow L^{n/(n-1),1}$ with the aid of the coarea formula in Section 5.
Here, $BV$ is the space of all functions of bounded variation.
Our estimate is covered by a result in Spector \cite{Spe}, although our argument is more succinct.

\vspace{5mm}

\section{Preliminary}
For a measurable subset $\Omega \subset \R^n$, let us denote the $n$-dimensional Lebesgue measure of $\Omega$ by $|\Omega|$.
The distribution function of a Lebesgue measurable function $f$ on $\Rn$ is defined  by 
\[
d_f(\lambda) := |\{x \in \R^n; |f(x)| > \lambda \}|,\ (\lambda \in [0,\infty)).
\]
It holds that $d_f(0) = |\{f \not = 0\}| \le |\tsupp f|$. 
In this article, we define $\mathbf{F}$ as the set of all Lebesgue measurable functions on $\Rn$ satisfying
\[
|f(x)| < \infty \ \tn{a.e.}\ x \in \Rn \ \tn{and} \ d_f(\lambda) < \infty,\ (\lambda \in (0,\infty)).
\]
Let us also define $\mathbf{F}_0$ as the space of all $f \in \mathbf{F}$ satisfying that for any bounded subset $E \subset \Rn$,
\[
d_{f \chi_E}(\lambda) \to 0 \ \tn{as}\ \lambda \to \infty.
\]
It is well-known that for $f \in \mathbf{F}$, the distribution function $d_f$ is right continuous on $(0,\infty)$.
The latter condition of $\mathbf{F}$ ensures this continuity.
In fact, if $f(x) := \lambda + |x|^{-1},\ (|x| \ge 1)$, then $d_f(\lambda + \vep) < \infty$ for any $\vep>0$, however $d_f(\lambda)=\infty$.
For $0<p<\infty$ and $0 < q \le \infty$, Lorentz spaces $L^{p,q}(\Rn)$ are defined by quasi-norms
\[
\|f\|_{L^{p,q}} :=
\begin{cases}
\; \left\{p \dint_0^\infty \left(\lambda d_f(\lambda)^{1/p} \right)^q \dfrac{d\lambda}{\lambda} \right\}^{1/q} \ & \tn{if} \ q<\infty \\
\; \dis \sup_{\lambda >0} \lambda d_f(\lambda)^{1/p} \ & \tn{if} \ q=\infty.
\end{cases}
\]

\vspace{5mm}

\subsection{Non-increasing rearrangements}
Although Lorentz quasi-norms are defined by the distribution function $d_f$, we prefer non-increasing rearrangements; for $f \in \mathbf{F}$ and $t \in [0,\infty)$
\begin{align*}
R_1[f](t) &:= \dis \inf\{\lambda \ge 0: d_f(\lambda) \le t \} \ \tn{and} \\
L_1[f](t) &:= \inf \{\lambda \ge 0: d_f(\lambda) < t \},
\end{align*}
where $\inf \emptyset := \infty$.
It is well-known that if $f \in \mathbf{F}_0$, then $R_1[f](t) < \infty$ for all $t \in (0,\infty)$, thus
\[
R_1[f](t) = \min \{\lambda \ge 0: d_f(\lambda) \le t \},\ (t \in (0,\infty)).
\]
Both $R_1[f]$ and $L_1[f]$ are equimeasurable with $|f|$, which means that
\[
d_f(\lambda) = d_{R_1[f]}(\lambda) = d_{L_1[f]}(\lambda),\ (\lambda \in (0,\infty)).
\]
It is also well-known that for $f \in \mathbf{F}_0$, the rearrangement $R_1[f]$ is right continuous, while $L_1[f]$ is left continuous.
We give a proof for the latter in Appendix, for the completeness.
We have other representations of them;
\begin{align*}
R_1[f](t) = R_2[f](t) &:= \inf_{|A| \le t} \|f\|_{L^\infty(A^c)},\ (t \in [0,\infty)), \ \tn{and} \\
L_1[f](t) = L_2[f](t) &:= \sup_{|A| = t} \ei_{A} |f| = L_2^\ast[f](t) := \sup_{|A| =t} \inf_A |f|,\ (t \in (0,\infty)).
\end{align*}
Here, of course,
\[
\dis \ei_A |f| := \sup \left\{\lambda>0; \left|\left\{A; |f| < \lambda \right\} \right| =0 \right\},
\]
where $\sup \emptyset := 0$.
One can see that $R_1[f](0) = R_2[f](0) = \|f\|_{L^\infty},\ L_1[f](0) = L_2[f](0) = \infty$ and $L^\ast_2[f](0) = \sup_{\Rn} |f|$.
See Proposition \ref{pointwise} below for the proof of these equalities.
The rearrangement $R_2$ can be found in Hyt\ho{o}nen \cite{H}, and $L_1$ and $L_2^\ast$ are applied by Lerner \cite{L1}, \cite{L2}, for example.
Based on such equalities, we write $R:= R_1=R_2$ and $L:=L_1=L_2=L_2^\ast$.
Lorentz spaces $L^{p,q}(\Rn)$ can be written as follows;
\[
\|f\|_{L^{p,q}} \approx
\begin{cases}
\; \left\{\dint_0^\infty \left(t^{1/p} N[f](t) \right)^q \dfrac{dt}{t} \right\}^{1/q}\ & \tnormal{if}\ q<\infty \\
\; \dis \sup_{t>0} t^{1/p} N[f](t) \ & \tnormal{if}\ q=\infty
\end{cases}
\]
where $N=R, L$.

\medskip

We show equalities above among rearrangements for the sake of completeness.

\begin{pro} \label{pro pointwise}
Let $f \in \mathbf{F}$ and $t \in (0,\infty)$.
Then one has
\[
R_1[f](t) = R_2[f] (t) \le L_1[f](t) = L_2[f](t) = L_2^\ast[f](t).
\]
\end{pro}

\medskip

\begin{rem}
\begin{enumerate}
\item
The first and second equality and first inequality remain to be true at $t=0$, as we see above.
If we drop the condition \llq$d_f(\lambda)< \infty,\ (\lambda > 0)$\rrq, then the second equality fails at $t=0$.
For example, if $f \equiv 1$, then $L_1[f](0) = \infty$ and $L_2[f](0)=1$.

\item
The opposite inequality $L[f](t) \le R[f](t),\ (t \in (0,\infty))$ fails in general. 
Step functions give counterexamples.

\item
We use a local analogy of the inequality of the form;
\[
R[f \chi_E](\alpha |E|) \le L[f \chi_E](\alpha |E|),\ (\alpha \in (0,1),\ 0<|E|<\infty).
\]
See Proposition \ref{L} for a representation of $L[f \chi_E](\alpha |E|)$.
\end{enumerate}
\end{rem}

\medskip

\begin{proof}
\underline{$\circ R_1[f](t) \le R_2[f](t)$:} We may assume that $R_2[f](t) < \infty$.
For any $\vep>0$ there is a measurable subset $A_\vep$ such that 
\[
\dis |A_\vep| \le t \ \& \ \infty > R_2[f](t) + \vep \ge \|f\|_{L^\infty(A_\vep^c)}.
\]
It also holds that $|\{|f| > \|f\|_{L^\infty(A_\vep^c)} \}| \le |A_\vep| \le t$.
Therefore,
\[
R_1[f](t) \le \|f\|_{L^\infty(A_\vep^c)} \le R_2[f](t) + \vep.
\]

\medskip

\underline{$\circ R_2[f](t) \le R_1[f](t)$:}
Since $|\{|f| > R_1[f](t)\}| \le t$, one has
\[
\dis R_2[f](t) = \inf_{|A| \le t} \|f\|_{L^\infty(A^c)} \le \|f\|_{L^\infty(\{|f| \le R_1[f](t)\})} \le R_1[f](t).
\]

\medskip

\underline{$\circ R_1[f](t) \le L_1[f](t)$ and $L_2^\ast[f](t) \le L_2[f](t)$:} These are immediate from definitions.

\medskip

\underline{$\circ L_1[f](t) \le L_2^\ast[f](t)$:}
In the case $L_1[f](t) = \infty$, because $d_f(\lambda) \ge t$ for any $\lambda >0$ and $(\Rn; dx)$ is a non-atomic space, there is a subset $A \subset \{|f| > \lambda \}$ such that $|A|=t$, which implies $L_2^\ast[f](t) \ge \lambda$.
Thus, $L_2[f](t) = L_2^\ast[f](t) = \infty$.
In the case $0<L_1[f](t) < \infty$, it holds $d_f(L_1[f](t) - \vep) \ge t$ for small $\vep > 0$.
Similarly as above, there exists a subset $B \subset \{|f| > L_1[f](t) - \vep \}$ so that $|B|=t$, which yields
\[
L_2^\ast[f](t) \ge \dis \inf_B |f|  \ge L_1[f](t) - \vep.
\]
Therefore, $L_1[f](t) \le L_2^\ast[f](t) \le L_2[f](t)$.

\medskip

\underline{$\circ L_2[f](t) \le L_1[f](t)$:}
In the case $\{\lambda \ge 0; d_f(\lambda)<t \} = \emptyset$, we see $L_1[f](t) = \infty$.
If not the case, we take $\lambda \in \{\lambda \ge 0; d_f(\lambda)<t \}$.
For any $A \subset \Rn$ with $|A| =t$, there is $x_A \in A$, for which $|f(x_A)| \le \lambda$.
Indeed, if $|f| >\lambda$ on $A$, then $|A| \le d_f(\lambda)<t$.
For such $A$, we have $|A \backslash \{|f| > \lambda \}| \ge t - d_f(\lambda) >0$, thus
\[
\dis \sup_{|A|=t} \ei_A |f| \le \sup_{|A| =t} \ei_{A \backslash \{|f| > \lambda \}} |f| \le \lambda.
\]
Hence, one has $L_2[f](t) \le \lambda$, which means that
\[
L_2[f](t) \le L_1[f](t).
\]
\end{proof}

\medskip

We end this subsection with a local analogy of equality between $L_1[f]$ and $L_2[f]$ in Proposition \ref{pro pointwise}, and a comparison of $L[f \chi_E]$ with fractional integral average.
\begin{pro} \label{L}
Let $f \in \mathbf{F},\ \gamma \in [0,n),\ E \subset \Rn$ with $0<|E|<\infty$ and $\alpha \in (0,|E|^{\gamma/n})$.\par
(1) 
\[
L[f \chi_E](\alpha |E|^{1-\gamma/n}) = \dis \sup_{\substack{|A|=\alpha |E|^{1-\gamma/n} \\ A \subset E}} \ei_A |f|.
\]

\medskip

(2)
\[
L[f\chi_E](\alpha |E|^{1-\gamma/n}) \le \inf_{r \in (0,\infty)} \left(\dfrac{|E|^{\gamma/n}}{\alpha |E|} \int_E |f|^r dx \right)^{1/r}. 
\]
\end{pro}

\begin{proof}
(1) We observe that
\begin{align*}
L_2[f \chi_E](\alpha |E|^{1-\gamma/n}) & = \dis \max \left(\sup_{\substack{|A|=\alpha |E|^{1-\gamma/n} \\ A \subset E}} \ei_A |f\chi_E|,\ \sup_{\substack{|A| = \alpha |E|^{1-\gamma/n} \\ A \cap E^c \not = \emptyset}} \ei_A |f \chi_E| \right) \\
& =: \max \left(X,Y \right).
\end{align*}
We verify $Y \le X$, to do this we take $A \subset \Rn$ with $|A| = \alpha |E|^{1-\gamma/n}$ and $A \cap E^c \not = \emptyset$.

\medskip

\underline{Case $|A \cap E^c| = 0$:} Because $|A\cap E| = |A| = \alpha |E|^{1-\gamma/n}$, it holds that
\begin{align*}
\ei_A |f \chi_E| & = \dis \sup \left\{\lambda>0; |\{A; |f \chi_E| < \lambda \}| =0 \right\} \\
& = \sup \left\{\lambda>0; |\{A \cap E; |f \chi_E| < \lambda \}|=0 \right\} \\
& = \ei_{A \cap E} |f| \\
& \le X.
\end{align*}

\medskip

\underline{Case $|A \cap E^c| >0$:} If $E_f := \{x \in \Rn; |f(x)| = \infty \}$, then $|E_f|=0$.
Hence, for any $\lambda \in (0,\infty)$,
\begin{align*}
\left|\{A; |f \chi_E| < \lambda \} \right| & = \left|\{A \cap E; |f \chi_E| < \lambda \} \right| + \left|\{A \cap E^c; |f \chi_E| < \lambda \} \right| \\
& =\left|\{A \cap E; |f \chi_E| < \lambda \} \right| + |\{A \cap E^c \cap E_f^c; |f \chi_E| < \lambda \} | \\
& = \left|\{A \cap E; |f \chi_E| < \lambda \} \right| + |A \cap E^c| >0.
\end{align*}
This means that $\left\{\lambda>0; |\{A; |f \chi_E| < \lambda \}| =0 \right\} = \emptyset$, and then $\ei_A |f \chi_E| =0 \le X$.

\medskip

(2) We have that for $A \subset E$ with $|A| = \alpha |E|^{1-\gamma/n}$
\[
\lan |f|^r \ran_E \ge \dfrac{|A|}{|E|} \ei_A |f|^r = \alpha |E|^{-\gamma/n} \ei_A |f|^r.
\]
This and (1) complete the proof.
\end{proof}

\vspace{5mm}

\section{Representation of $\med^{\alpha,\gamma}_E[f]$}
Here, we give a representation of $\med_E^{\alpha,\gamma}[f]$ for non-negative functions $f \in \mathbf{F}_0$.
The inclusion \llq $\subset$\rrq in Theorem \ref{express fra med} below with $\gamma=0$ was proved by Poelhuis-Torchinsky \cite{P-T}.
While our proof is very elementary, for the completeness, we give the proof. 

\begin{thm} \label{express fra med}
Let $f \in \mathbf{F}_0,\ E \subset \Rn$ with $0<|E|<\infty,\ \gamma \in [0,n)$ and $\alpha \in (0,|E|^{\gamma/n})$.
If $f$ is non-negative, then it holds that
\[
\med^{\alpha,\gamma}_E[f] = \left[R[f \chi_E]\left(\alpha |E|^{1-\gamma/n} \right), L[f \chi_E]\left(\alpha |E|^{1-\gamma/n} \right) \right].
\]
\end{thm}

\begin{proof}
Let $A$ and $B$ be defined as follows:
\begin{align*}
A & := \left\{ m \in \R; \left|\left\{x \in E; f(x) < m \right\} \right| \le |E| - \alpha |E|^{1-\gamma/n}  \right\} \ \tn{and} \\
B &:= \left\{ m \in \R; \left|\left\{x \in E; f(x) > m \right\} \right| \le \alpha |E|^{1-\gamma/n}\right\}.
\end{align*}
It is clear that $\med_E^{\alpha, \gamma}[f] = A \cap B$. 
Then, from the definition, it obviously holds $\med_E^{\alpha, \gamma}[f] = A \cap B$.
One can see that $A$ and $B$ are the one of the forms
\[
(-\infty, a)\ \tn{or} \ (-\infty,a] \ \tn{or}\ \R \ \tn{or}\ \emptyset,
\]
and
\[
(b,\infty)\ \tn{or}\ [b,\infty) \ \tn{or}\ \R\ \tn{or}\ \emptyset,
\]
for some $a,b \in \R$, respectively.

Firstly, we show that
\begin{equation} \label{AB}
\med_E^{\alpha,\gamma}[f] = [\max A, \min B]
\end{equation}
where $\max A, \min B \in \R$.
We divide the proof into several steps.

\medskip

\ul{Step 1: $A \not = \emptyset$ and $B \not = \emptyset$.}
If $A = \emptyset$, it holds that
\[
\dis 0=|\{E; f=-\infty \}| = \left|\bigcap_{\ell \in \N} \{E; f<-\ell \} \right| = \lim_{\ell \to \infty} |\{E; f<-\ell \}| \ge |E| - \alpha |E|^{1-\gamma/n} > 0.
\]
Hence, $A \not = \emptyset$.
Similarly, if we assume $B = \emptyset$, we get the following contradiction;
\[
0 = |\{E; f=\infty \}| = \left|\bigcap_{\ell \in \N} \{E; f>\ell \} \right| = \lim_{\ell \to \infty} |\{E; f > \ell \}| \ge \alpha |E|^{1-\gamma/n} > 0.
\]

\medskip

\ul{Step 2: $A \not = \R$ and $B \not = \R$:}
If $A = \R$, one has that
\[
\dis |E| = |\{E; f< \infty \}| = \left|\bigcup_{\ell \in \N} \{E; f < \ell \} \right| = \lim_{\ell \to \infty} |\{E; f<\ell \}| \le |E| - \alpha |E|^{1-\gamma/n}.
\]
This means $\alpha \le 0$.
Then, we can see that $A = (-\infty,a)$ or $A=(-\infty,a]$ for some $a \in \R$.
Hence, $\sup A \in \R$.
Similarly, if $B=\R$, then we have the contradiction 
\[
|E| = |\{E; f >-\infty\} = \left|\bigcup_{\ell \in \N} \{E; f > -\ell \} \right| = \lim_{\ell \to \infty} |\{E; f > -\ell \}| \le \alpha |E|^{1-\gamma/n}.
\]
Thus, we have $\inf B \in \R$.

\medskip

\ul{Step 3: $\sup A \in A \ \& \ \inf B \in B$.}
These are trivial from step 2;
\[
\dis |\{E; f< \sup A \}| = \lim_{\ell \to \infty} |\{E; f < \sup A - 1/\ell\}| \le |E| - \alpha |E|^{1-\gamma/n},
\]
and
\[
\dis |\{E; f > \inf B\}| = \lim_{\ell \to \infty} |\{E; f> \inf B + 1/\ell \}| \le \alpha |E|^{1-\gamma/n}.
\]
Therefore, $\sup A = \max A \in \R$ and $\inf B = \min B \in \R$.

\medskip

\ul{Step 4: $A \cap B \not = \emptyset$.}
More precisely, we prove that $\max A \in B \ \& \ \min B \in A$, which means that $\med_E^{\alpha, \gamma}[f] = [\max A, \min B]$.

\medskip

These are easy to see as follows;
\begin{align*}
\dis |\{E; f> \max A \}| & = \left|\bigcup_{\ell \in \N} \left\{E; f \ge \max A + 1/\ell \right\} \right| \\
& = \lim_{\ell \to \infty} |\{E; f \ge \max A + 1/\ell \}| \\
& = |E| - \lim_{\ell \to \infty} |\{E; f < \max A + 1/\ell \}| \\
& \le |E| - \left(|E| - \alpha |E|^{1-\gamma/n} \right) \\
& = \alpha |E|^{1-\gamma/n}
\end{align*}
and
\begin{align*}
|\{E; |f| < \min B \}| & = \left|\bigcup_{m \in \N} \left\{E; |f| \le \min B - 1/m \right\} \right| \\
& = \dis \lim_{m \to \infty} \left|\left\{E; |f| \le \min B -1/m \right\} \right| \\
& = |E| - \lim_{m \to \infty} \left|\left\{E; |f| > \min B -1/m \right\} \right| \\
& \le |E| - \alpha |E|^{1-\gamma/n}.
\end{align*}

\medskip

Next, we show that for any $m \in \med_E^{\alpha,\gamma}[f]$,
\begin{equation} \label{RmL}
R[f\chi_E](\alpha |E|^{1-\gamma/n}) \le m \le L[f \chi_E](\alpha |E|^{1-\gamma/n}).
\end{equation}
We should emphasize that this inequality in the case $\gamma=0$ were proved by Poelhuis-Torchinsky \cite{P-T}.
For simplicity, we write $R:= R[f\chi_E](\alpha |E|^{1-\gamma/n})$ and $L:=L[f \chi_E](\alpha |E|^{1-\gamma/n})$.
Since
\[
d_{f \chi_E}(m) = \left|\left\{E; f > m \right\}\right| \le \alpha |E|^{1-\gamma/n},
\]
we see $R \le m$.
To show the second inequality, we see for any $\vep>0$ that
\[
|\{E; |f| \le L+\vep \}| = |E| - d_{f \chi_E}(L+\vep) > |E| - \alpha |E|^{1-\gamma/n}.
\]
Now we assume that $L+\vep < m$, then one has the contradiction as follows:
\[
|E| - \alpha |E|^{1-\gamma/n} < |\{E; |f| \le L+\vep \}| \le |\{E; |f| < m \}| \le |E| - \alpha |E|^{1-\gamma/n}.
\]
Thus, $m \le L$.
We have proved that
\[
\med_E^{\alpha,\gamma}[f] = [\max A, \min B] \subset \left[R[f\chi_E](\alpha |E|^{1-\gamma/n}), L[f \chi_E](\alpha |E|^{1-\gamma/n}) \right].
\]

\medskip

Finally, we show that $R, L \in \med_E^{\alpha, \gamma}[f]$ by using the assumption $f \in \mathbf{F}_0$.
On one hand, it holds that
\begin{align*}
\left| \left\{E; f < R \right\} \right| & \le \left| \left\{E; f < L \right\} \right| = \left|\bigcup_{m=1}^\infty \left\{E; f \le L - \frac{1}{m} \right\} \right| \\
& = |E| - \lim_{m \to \infty} \left|\left\{E; f > L - \frac{1}{m} \right\} \right| \\
& = |E| - \lim_{m \to \infty} d_{f\chi_E} \left(L - \frac{1}{m} \right) \le |E| - \alpha |E|^{1-\gamma/n},
\end{align*}
thus $R,L \in A$.
On the other hand, one has
\begin{align*}
\left| \left\{E; f > L \right\} \right| & \le \left| \left\{E; f > R \right\} \right| = \left|\bigcup_{m=1}^\infty \left\{E; f> R + \frac{1}{m} \right\} \right| \\
& = \lim_{m \to \infty} d_{f\chi_E} \left(R+\frac{1}{m} \right) \le \alpha |E|^{1-\gamma/n},
\end{align*}
thus $R, L \in B$.
\end{proof}

\vspace{5mm}

\section{Fractional maximal operators}
In this section, let $\Omega$ be a non-empty open subset in $\Rn$.
The purpose of this section is to define fractional maximal operators with non-increasing rearrangements $R$ and $L$, and to give mapping properties of them.

\medskip

\begin{Def} \label{frac max}
Let $f \in \mathbf{F},\ \alpha \in (0,\infty)$ and $\gamma \in [0,n)$.
We define the fractional maximal operators
\[
\begin{cases}
\; m_{\alpha, \Omega, \gamma}^R[f](x) := \dis \sup_{Q \subset \Omega} R[f \chi_Q]\left(\alpha |Q|^{1-\gamma/n} \right) \chi_Q(x)\ \tn{and} \\
\; m_{\alpha, \Omega, \gamma}^L[f](x) := \dis \sup_{Q \subset \Omega} L[f \chi_Q]\left(\alpha |Q|^{1-\gamma/n} \right) \chi_Q(x).
\end{cases}
\]
In the case $\Omega=\Rn$ or $\gamma=0$, we omit them.
\end{Def}

\medskip

\begin{rem} \label{max fra rem}
\begin{enumerate}
\item
The maximal operators $m^R_{1/2}$ and $m^L_{1/2}$ were introduced by Strömberg \cite{St}.

\item
It follows from Propositions \ref{pro pointwise} and \ref{L} that
\[
m^R_{\alpha, \Omega, \gamma}[f](x) \le m^L_{\alpha, \Omega, \gamma}[f](x) \le \inf_{r \in (0,\infty)} \left(\alpha^{-1/r} M^\Omega_\gamma (|f|^r)(x)^{1/r} \right),
\]
where $M^\Omega_\gamma$ is the fractional maximal operator, defined by
\[
M^\Omega_\gamma f(x) := \dis \sup_{Q \subset \Omega} |Q|^{\gamma/n} \lan f \ran_Q \chi_Q(x).
\]
We write $M_\gamma = M^{\Rn}_\gamma$.

\item
The fractional maximal operator $M^\Omega_\gamma$ is written by the rearrangements; for $N= R$ or $L$,
\[
M^\Omega_\gamma f(x) = \sup_{Q \subset \Omega} \int_0^{|Q|^{\gamma/n}} N[f \chi_Q] (\alpha |Q|^{1-\gamma/n}) d\alpha \chi_Q(x).
\]

\item
In fact, the suprema in Definition \ref{frac max} are restricted to cubes whose volume is larger than $\alpha^{n/\gamma}$.
For $\alpha \ge |Q|^{\gamma/n}$,
\[
m_{\alpha, \Omega, \gamma}^R[f](x) = m_{\alpha, \Omega, \gamma}^L[f](x) = 0,\ (x \in \Rn)
\]
because
\begin{align*}
R[f \chi_Q](\alpha |Q|^{1-\gamma/n}) & \le R[f \chi_Q](|Q|) = 0,\ \tn{and} \\
L[f \chi_Q](\alpha |Q|^{1-\gamma/n}) & \le L[f \chi_Q](|Q|) = 0. \\
\end{align*}

\item
Fujii \cite{F2} in the case $\alpha=1/2$ and Poelhuis-Torchinsky \cite{P-T} in general case proved that for a real-valued, finite a.e. measurable function $f$ on $\Rn$ and $\alpha \in (0,1)$, it holds that
\[
\lim_{x \in Q_j,\ Q_j \searrow \{x\}} m_{Q_j} = f(x) \ \tn{a.e.}\ x \in \Rn
\]
with any $m_{Q_j} \in \med_{Q_j}^{\alpha,0}[f]$.
From this result, we see that
\begin{equation} \label{pointwise}
|f(x)| \le m^R_{\alpha}[f](x)\ \tn{for almost all}\ x \in \Rn.
\end{equation}
This is applied in Section 5.

\item
We can define the maximal operator via fractional medians similar to that above.
But the maximal operator is not unique, because these are many choices of fractional medians for each cubes.
Moreover, they are controlled from both below and above by the rearrangements.
\end{enumerate}
\end{rem}

\medskip

\subsection{Boundedness of the fractional maximal operators}
In this subsection, we prove the following bounds.
\begin{thm} \label{maximal}
Let $p \in (0,\infty),\ q \in (0,\infty],\ \gamma \in [0,n),\ r \in [1,n/\gamma)$ and $1/r - 1/\wti{r} = \gamma/n$.
Then
\[
\left\|m^L_{\alpha, \gamma} \right\|_{L^{p,q} \to L^{\wti{p},q}} \le \left(\dfrac{\|M_\gamma\|_{L^{r,1} \to L^{\wti{r},\infty}}}{\alpha} \right)^{r/p},
\]
where $\wti{p} = (p\wti{r}/r)$, equivalently
\[
\dfrac{1}{p} - \dfrac{1}{\wti{p}} = \dfrac{r}{p} \dfrac{\gamma}{n}.
\]
\end{thm}

\medskip

\begin{rem} \label{rem bdd}
Because $r$ is independent of $p$ and $q$, Theorem \ref{maximal} implies that $m^L_{\alpha,\gamma}: L^{p,q} \to L^{\wti{p},q}$ for all $\wti{p} \in [np/(n-\gamma), \infty)$.
Here, we note that $\wti{p}$ is not fixed, which contrasts with the case for $M_\gamma$.
This phenomenon can be understood from 2 in Remark \ref{max fra rem}.
Since $\wti{p}$ can be larger than $np/(n-\gamma)$, the operator $m^L_{\alpha,\gamma}$ is smoother than $M_\gamma$.
\end{rem}

\medskip

The ingredient of the proof is an interesting equality or inclusions between the (strict) super-level sets of $m^R_{\alpha, \gamma}[f]$ and $M_\gamma(\chi_{\{\Omega; |f|>\lambda\}})$.

\begin{lem} \label{mainlem fra}
Let $f \in \mathbf{F}_0,\ \gamma \in [0,n),\ \alpha >0$ and $\lambda > 0$.
Then,\par
(1)
\[
\dis \left\{x \in \Omega; m_{\alpha,\Omega, \gamma}^R[f](x) > \lambda \right\} = \left\{x \in \Omega; M^\Omega_\gamma \left(\chi_{\{\Omega; |f| > \lambda\}} \right)(x) > \alpha \right\}.
\]

\medskip

(2)
\begin{align*}
\left\{x \in \Omega; m^L_{\alpha,\Omega, \gamma}[f](x) > \lambda \right\} & \subset \left\{x \in \Omega; M^\Omega_\gamma(\chi_{\{\Omega; |f| > \lambda \} })(x) \ge \alpha \right\} \\
& = \bigcap_{\vep>0} \left\{x \in \Omega; m^R_{\alpha-\vep,\Omega,\gamma}[f](x) > \lambda \right\} \\
& = \bigcap_{\vep>0} \left\{x \in \Omega; m^L_{\alpha-\vep,\Omega,\gamma}[f](x) > \lambda \right\}.
\end{align*}
\end{lem}

\medskip

\begin{rem} \label{rem inclusions}
\begin{enumerate}
\item
From (\ref{pointwise}), (1) in Lemma \ref{mainlem fra} and (2) in Proposition \ref{L}, we have
\begin{align*}
d_f(\lambda) & \le d_{m^R_{\alpha}[f]}(\lambda) = d_{M(\chi_{|f| > \lambda})}(\alpha) \\
& \le \inf_{p \in [1,\infty)} \left(\dfrac{\|M\|_{L^{p,1} \to L^{p,\infty}}}{\alpha} \|\chi_{|f| > \lambda}\|_{L^{p,1}} \right)^p \le c d_f(\lambda).
\end{align*}
This implies that $\|f\|_{L^{p,q}} \approx \|m^R_{\alpha}[f]\|_{L^{p,q}}$ for all $p,q$.
A similar equivalence holds for $m^L_{\alpha}$ as shown in (2).
The inequalities above mean that $m^R_{\alpha}[f]$ is \llq almost\rrq equimeasurable with $f$.
However, the operator $m^R_\alpha$ connects $f$ to its super-level set in terms of the distribution function.
The author does not know of any pointwise inequality between $f$ and $M(\chi_{\{|f| > \lambda \}})$.
Although $M(\chi_{\{|f|>\lambda \}}) \le 1$, $f$ can be large.

\item
We note that the roles of $\alpha$ and $\gamma$ are swapped in (1) and (2) in Lemma \ref{mainlem fra} below.
The equality in (1) of Lemma \ref{mainlem fra} for the case $\gamma=0$ is found in many literatures, and the proof for other cases is almost same.
For the rearrangement $L$, although such equality does not hold, the super-level sets are equivalent in some sense.
We will apply (1) and (2) in next section
\end{enumerate}
\end{rem}

\medskip

\begin{proof}
(1) We shall show the inclusion \llq$\subset$\rrq.
The opposite inclusion can be seen by tracing the argument below backwards.
If $m^R_{\alpha, \Omega, \gamma}[f](x) > \lambda$, then there exists $\Omega \supset Q \ni x$ so that $R[f \chi_Q](\alpha |Q|^{1-\gamma/n}) > \lambda$.
(Here, the cube $Q$ must satisfies $\alpha < |Q|^{\gamma/n}$. See Remark \ref{max fra rem}.)
From the definition of the rearrangement $R=R_1$, it follows $d_{f \chi_Q}(\lambda) > \alpha |Q|^{1-\gamma/n}$.
Thus
\[
|Q|^{\gamma/n} \dfrac{\int_Q \chi_{\{|f| > \lambda \}} dx}{|Q|} > \alpha,
\]
which means that $M_\gamma (\chi_{\{\Omega; |f| > \lambda \}})(x) > \alpha$.

\medskip

(2) The first inclusion follows from the same argument in (1).
To show the first equality, we take $x$ so that $M_\gamma(\chi_{\{\Omega; |f| > \lambda \}})(x) \ge \alpha$.
Then, for any $\vep>0$ there is $x \in Q_\vep \subset \Omega$, for which $|Q_\vep|^{\gamma/n-1} \int_{Q_\vep} \chi_{\{\Omega; |f| > \lambda \}} dx > \alpha - \vep$.
This is same as the inequality $d_{f \chi_{\{\Omega; |f| > \lambda \}}}(\lambda) > (\alpha - \vep) |Q_\vep|^{1-\gamma/n}$, thus
\[
\lambda < R\left[f \chi_{\{\Omega; |f| > \lambda \}} \right]\left((\alpha-\vep) |Q_\vep|^{1-\gamma/n} \right) \le m^R_{\alpha-\vep,\gamma}[f](x).
\]
To trace this argument backwards, the opposite inclusion can be seen.
The same argument verifies the second equality.
\end{proof}

\medskip

We give the proof of Theorem \ref{maximal}.
\begin{proof}
It is easy to see that
\[
d_{m^L_{\alpha,\gamma}[f]}(\lambda) \le \left(\dfrac{\|M_\gamma\|_{L^{r,1} \to L^{\tilde{r},\infty}}}{\alpha} \right)^{\tilde{r}} d_f(\lambda)^{\tilde{r}/r}.
\]
Indeed, from (2) in Lemma \ref{mainlem fra}, we have
\begin{align*}
d_{m^L_{\alpha,\gamma}[f]}(\lambda) & \le \lim_{\vep \to 0} d_{M_\gamma(\chi_{\{|f| > \lambda \}})}(\alpha - \vep) \\
& \le \alpha^{-\tilde{r}} \|M_\gamma \|^{\tilde{r}}_{L^{r,1} \to L^{\tilde{r},\infty}} \|\chi_{\{|f| > \lambda \}}\|_{L^{r,1}}^{\tilde{r}} \\
& = \alpha^{-\tilde{r}} \|M\|^{\tilde{r}}_{L^{r,1} \to L^{\tilde{r},\infty}} d_f(\lambda)^{\tilde{r}/r}.
\end{align*}
Therefore, in the case $q<\infty$,
\begin{align*}
\left\|m^L_{\alpha, \gamma}[f] \right\|_{L^{\wti{p},q}} & = \left(\wti{p} \int_0^\infty \left(\lambda d_{m^L_{\alpha,\gamma}[f]}(\lambda)^{1/\wti{p}} \right)^q \dfrac{d\lambda}{\lambda} \right)^{1/q} \\
& \le \left(\dfrac{\wti{p}}{p} \right)^{1/q} \left(\dfrac{\|M_\gamma\|_{L^{r,1} \to L^{\tilde{r},\infty}}}{\alpha} \right)^{\tilde{r}/\wti{p}} \|f\|_{L^{p,q}}.
\end{align*}
The case $q=\infty$ is similar.
\end{proof}

\vspace{5mm}

\section{Another proof of Alvino's embedding $BV \hookrightarrow L^{n/(n-1),1}$}
We shall give a proof of Alvino's embedding $BV \hookrightarrow L^{n/(n-1),1}$ by using medians.

\medskip

The Sobolev embedding $W^{1,1} \hookrightarrow L^{n/(n-1)},\ (n \ge 2)$ is well-known.
This result was improved by Alvino \cite{A} who showed that $BV \hookrightarrow L^{n/(n-1),1}$, where $BV$ denotes the space of all functions of bounded variation.
\begin{Def}
For $f \in L^1_{loc}$, we define
\[
|f|_{BV} := \dis \sup \left\{\int f \tn{div} \Phi dx; \Phi \in C^\infty_0(\Rn; \Rn) \ \tn{with}\ \|\Phi\|_{L^\infty} \le 1 \right\},
\]
and $\|f\|_{BV} := \|f\|_{L^1} + |f|_{BV}$.
A subset $E \subset \Rn$ is called a Cacciopolli set if $|\chi_E|_{BV} < \infty$.
\end{Def}

\medskip

\begin{rem}
\begin{enumerate}
\item
While $W^{1,1} \subsetneq BV$, it holds that $|f|_{BV} = \|\nabla f\|_{L^1}$ for $f \in W^{1,1}$.

\item
We apply the coarea formula for functions in $BV$ in Subsection 4.2;
\begin{equation} \label{coarea}
|f|_{BV} = \int_\R |\chi_{E_\lambda}|_{BV} d\lambda,
\end{equation}
where $E_\lambda := \{x \in \Rn; f(x) > \lambda\}$.
This formula is found in several literatures, for example Evans-Gariepy \cite{E-G}.
\end{enumerate}
\end{rem}

\medskip

We shall show that for $n \ge 2,\ \gamma \in (0,n-1)$ and $f \in BV$,
\begin{equation} \label{Alvino}
\left\|m^L_{\alpha,\gamma}[f] \right\|_{L^{n/(n-(1+\gamma)),1}} \lesssim \|f\|_{L^{n/(n-1),1}} \lesssim |f|_{BV},
\end{equation}
where for the second inequality we assume that $f$ is non-negative.
Spector \cite{Spe} proved more strong inequalities.

\medskip

\textit{Proof of \eqref{Alvino}:}
Since
\[
\dfrac{n-1}{n} - \dfrac{n-(1+\gamma)}{n} = \dfrac{\gamma}{n},
\]
the first inequality follows from Theorem \ref{maximal} with $r=p=n/(n-1)$.

\medskip

We shall show the second inequality.
From 1 of Remark \ref{rem inclusions}, we have that
\begin{align*}
\|f\|_{L^{n/(n-1),1}} & \approx \left\| m^R_{\alpha}[f] \right\|_{L^{n/(n-1),1}} = \int_0^\infty d_{M(\chi_{\{f > \lambda\}})}(\alpha)^{(n-1)/n} d\lambda \\
& \le \alpha^{-1} \int_0^\infty \left\|M(\chi_{\{f >\lambda \}})\right\|_{L^{n/(n-1),\infty}} d\lambda.
\end{align*}
Since $\chi_{\{f>\lambda \}} \in L^{n/(n-1)}$ and $M$ is bounded on $L^{n/(n-1),\infty}$, we have
\[
\left\|M(\chi_{\{f >\lambda \}})\right\|_{L^{n/(n-1),\infty}} = \dis \lim_{\vep \to 0} \left\|M(\vphi_\vep \ast \chi_{\{f >\lambda \}})\right\|_{L^{n/(n-1),\infty}},
\]
where $\vphi_\vep$ is the standard mollifier.
Next, we use a standard equality; for $g \in C^1(\Rn)$ satisfying $g(x) \to 0$ as $|x| \to \infty$
\[
g(x) = \dfrac{1}{|\SP^{n-1}|} \int_\Rn \dfrac{(x-y)}{|x-y|^n} \cdot \nabla g(y) dy,
\]
which implies $|g| \lesssim I_1(|\nabla g|)$ where $I_\gamma g(x) := g \ast |\cdot|^{-(n-\gamma)}(x)$.
Combining with an interesting pointwise bound from Hong-Moen-P\'erez \cite{H-M-P}
\begin{equation} \label{MI0}
Mg(x) \lesssim MI_1(|\nabla g|)(x) \lesssim I_1(|\nabla g|)(x),
\end{equation}
one obtains that from $I_1: L^1 \to L^{n/(n-1),\infty}$,
\begin{align*}
\|f\|_{L^{n/(n-1),1}} & \lesssim \int_0^\infty \limsup_{\vep \to 0} \left\|I_1 \left(\left|\nabla \left(\vphi_\vep \ast \chi_{\{f > \lambda \}} \right) \right| \right) \right\|_{L^{n/(n-1),\infty}} d\lambda \\
& \lesssim \int_0^\infty \limsup_{\vep \to 0} \left|\vphi_\vep \ast \chi_{\{f > \lambda \}} \right|_{BV} d\lambda.
\end{align*}
The convergence of the mollifier in $|\cdot|_{BV}$, see Remark 1.16 in Giusti \cite{Gu}, and the coarea formula (\ref{coarea}) complete the proof;
\[
\|f\|_{L^{n/(n-1),1}} \lesssim \int_0^\infty |\chi_{\{f>\lambda \}}|_{BV} d\lambda \le \int_\R |\chi_{\{f>\lambda \}}|_{BV} d\lambda = |f|_{BV}.
\]
\qed

\medskip

As mentioned in Introduction, our argument is covered by the result $\|I_\gamma \nabla f \|_{L^{n/(n-\gamma),1}(\Rn;\Rn)} \lesssim \|\nabla f \|_{L^1(\Rn;\Rn)}$ from Spector \cite{Spe}.

\vspace{5mm}

\section{Appendix: Left continuity and equimeasurability of $L$}
Here, we prove that the left continuity of $L[f]$ and the equimeasurability of $L[f]$ with $f$.
\begin{pro} \label{left conti}
Let $f \in \mathbf{F}$ and $t \in (0,\infty)$.
If $L[f](t) < \infty$, then $L[f]$ is left continuous at $t$.
\end{pro}

\begin{proof}
From the definition of $L_1[f]$, we see that $d_f(L_1[f](t) + \vep) < t$ for any $\vep >0$.
Thus, there is $\delta \in (0,t)$ so that
\[
d_f(L_1[f](t) + \vep) < t - \delta.
\]
This means that $L_1[f](t - \delta) \le L_1[f](t) + \vep$.
\end{proof}

\medskip

\begin{pro}
If $f \in \mathbf{F}$, then $L[f]$ is equimeasurable with $f$;
\[
d_f(\lambda) = d_{L[f]}(\lambda)\ \tnormal{for}\ \lambda \in (0,\infty).
\]
\end{pro}

\begin{proof}
We will prove this for $L_2[f]$.

Fix $\lambda \in [0,\infty)$ and denote
\[
\Omega_\lambda := \{t \in [0,\infty); L_2[f](t) > \lambda \},
\]
thus $\Omega_\lambda$ is one of $\emptyset,\ (0,T),\ (0,T]$ or $(0,\infty)$ with some $T \in (0,\infty)$.

\medskip

\noindent
$\circ$ \underline{Case: $\Omega_\lambda = \emptyset$.}
In this case, since $d_{L_2[f]}(\lambda) = 0$ it suffers to show $d_f(\lambda)=0$.
We assume that $d_f(\lambda) \in (0,\infty)$.
Then, from the right continuity of the distribution function, there are $\tau \in (0,\infty)$ and $\vep \in (0,\infty)$ such that $d_f(\lambda + \vep) > \tau$.
We can find a measurable subset $A \subset \{|f| > \lambda + \vep \}$ satisfying $|A| = \tau$, which implies $L_2[f](\tau) \ge \lambda + \vep$.
This contradicts $\Omega = \emptyset$.

\medskip

\noindent
$\circ$ \underline{Case: $\Omega_\lambda = (0,T)$ or $(0,T]$.}
In this case, $T = \dis \sup \Omega_\lambda \in (0,\infty)$.
For any $\vep \in (0, T^{-1})$, it holds $L_2[f](T - \vep) > \lambda$.
Hence, there exists $A_\vep \subset \Rn$ so that $|A_\vep| = T - \vep$ and $\dis \inf_{x \in A_\vep} |f(x)| > \lambda$.
Thus $d_f(\lambda) \ge |A_\vep| = T - \vep$, which means $d_f(\lambda) \ge T$.
If we assume that $d_f(\lambda) > T$, then there exist $t_0 \in (T, d_f(\lambda))$ and $\delta > 0$ such that $d_f(\lambda + \delta) > t_0$.
Since we can find $B_\delta \subset \{|f| > \lambda + \delta \}$ fulfilling $|B_\delta| = t_0$, one has
\[
L_2[f](t_0) \ge \dis \inf_{x \in B_\delta} |f(x)| \ge \lambda + \delta,
\]
and then one has the contradiction $t_0 \le T$.
Therefore, $d_f(\lambda) = T = |\Omega_\lambda| = d_{L_2[f]}(\lambda)$.

\medskip

\noindent
$\circ$ \underline{Case: $\Omega_\lambda = (0,\infty)$.}
Since $d_{L[f]}(\lambda) = \infty$, it is enough to show that $d_f(\lambda) \ge t$ for any $t \in (0,\infty)$.
Because $L_2[f](t) > \lambda$ for $t \in (0,\infty)$, we have measurable subsets $\{A_t\}_{t \in (0,\infty)}$ enjoying $|A_t| = t$ and $\dis \inf_{x \in A_t} |f(x)| > \lambda$, which yields $d_f(\lambda) \ge |A_t| = t$.
\end{proof}

\medskip

\section*{Acknowledgement}
The author would like to thank the referee for carefully reading.
This work was supported by JSPS KAKENHI Grant Numbers JP18KK0072, JP19K03538, JP20H01815, JP23K03181.

\bibliography{sn-bibliography}

\end{document}